\documentclass[12pt,leqno,english]{amsart}
\usepackage{mathptmx}
\usepackage{helvet}
\usepackage{courier}
\usepackage[latin9]{inputenc}
\usepackage{textcomp}
\usepackage{mathrsfs}
\usepackage{mathtools}
\usepackage{amsbsy}
\usepackage{amstext}
\usepackage{amsthm}
\usepackage{amssymb}

\makeatletter
\numberwithin{equation}{section}
\numberwithin{figure}{section}

\@ifundefined{date}{}{\date{}}
\usepackage{tikz-cd}
\usepackage[english]{babel}
\usepackage[all]{xy}
\usepackage{mathrsfs}
\usepackage{enumitem}
\newtheorem{theorem}{Theorem}[section]
\newtheorem{lemma}[theorem]{Lemma}\newtheorem{corollary}[theorem]{Corollary}\newtheorem{proposition}[theorem]{Proposition}\newtheorem{definition}[theorem]{Definition}\newtheorem{remark}[theorem]{Remark}\numberwithin{equation}{section}

\usepackage{ulem}

\theoremstyle{plain}
\newtheorem*{remark*}{Remark}

\makeatother

\usepackage{babel}
\begin{document}

\title{On The Spectral Zeta Function Of Second Order Semiregular Non-Commutative
Harmonic Oscillators}
\author{Marcello Malagutti}
\address{Department of Mathematics, University of Bologna, Piazza di Porta
S.Do\-na\-to 5, 40126 Bologna, ITALY}
\email{marcello.malagutti2@unibo.it}

\maketitle

\begin{abstract}
In this paper we give a meromorphic continuation of the spectral zeta
function for \textit{semiregular} Non-Commutative Harmonic Oscillators
(NCHO). By ``semiregular system'' we mean systems with terms with
degree of homogeneity scaling by $1$ in their asymptotic expansion.
As an application of our results, we first compute the meromorphic
continuation of the Jaynes-Cummings (JC) model spectral zeta function.
Then we compute the spectral zeta function of the JC generalization
to a $3$-level atom in a cavity. For both of them we show that it
has only one pole in $1$.
\end{abstract}

\tableofcontents{}

\global\long\def\theequation{\thesection.\arabic{equation}}%


\section{Introduction}

One of the most important observables of the spectrum of an elliptic
operator is the spectral zeta function. For a complex Hilbert space
$H$ and a densely defined linear operator $P:H\rightarrow H$, we
denote the set of the eigenvalues (repeated by multiplicity) of $P$
by $\mathrm{Spec}\,P$. When $\mathrm{Spec}\,P$ is discrete we can
define the spectral zeta function of $P$ as
\[
\zeta_{P}(s):={\displaystyle \sum_{\lambda\in\mathrm{Spec}\,P}}\lambda^{-s},
\]
for any given complex number $s$ for which it makes sense. In particular,
if $P$ is an elliptic, selfadjoint and positive global pseudodifferential
operator of order $\mu>0$ on $\mathbb{R}^{n}$, then $s\mapsto\zeta_{P}(s)$
is holomorphic for $\mathrm{Re}s>2n/\mu$ since the defining series
is absolutely convergent (see Corollary 4.4.4. in \cite{P1:2010}).
For instance, if we denote by $P=\frac{x^{2}-\partial_{x}^{2}}{2}$
the harmonic oscillator defined as the maximal operator in $L^{2}(\mathbb{R})$,
then $\mathrm{Spec}\,P=\{k+1/2;\:k\in\mathbb{Z}_{+}\}$ with multiplicity
$1$, and
\[
\zeta_{P}(s)={\displaystyle \sum_{k\geq0}(k+1/2)^{-s}=(2^{s}-1)\zeta(s)},
\]
where $\zeta(s)$ denotes the Riemann zeta function. Note that $\zeta_{P}$
is holomorphic for $\mathrm{Re}(s)>1$, and has a meromorphic continuation
to the whole complex plane. Furthermore, $\zeta_{P}$ has the only
pole at $s=1$, and we have $\zeta_{P}(s)=0$ for $s=-2k$, $k\in\mathbb{Z}_{+}$
which are, thus, called trivial zeros. Moreover, the spectral zeta
function entangles information about the spectrum of $P$ in its analytical
properties. For instance, the residues of the zeta function at its
poles gives the coefficients of the Weyl Law for $P$ by the Ikehara
Tauberian theorem (see Section 14 of Shubin \cite{SH:2001}. See also
Proposition (IV.6) in \cite{HR:1980} and the references in Ivrii
\cite{Iv:2016}).

The notion of spectral zeta function was introduced for the first
time for the Laplacian on a two-dimensional Euclidean domains $\Omega$
by Carleman \cite{C:1934} who studied the Dirichlet-type series
\begin{equation}
{\displaystyle \sum_{\lambda_{j}\in\mathrm{Spec}\,\Delta}}\frac{\phi_{\lambda_{j}}(x_{1})\phi_{\lambda_{j}}(x_{2})}{\lambda_{j}^{s}},\;x_{1},x_{2}\in\Omega\label{eq:DirichletSeries}
\end{equation}
where $\phi_{\lambda_{j}}$ is the eigenfunction of $\Delta$ associated
to the eigenvalue $\lambda_{j}$. Later, in the case of a bounded
Euclidean domain $V$ of arbitrary dimension $N,$ Minakshisundaram
\cite{M:1949} showed through a method different from Carleman's that
(\ref{eq:DirichletSeries}) is an entire function of $s$ with zeros
at negative integers and that
\[
{\displaystyle \sum_{\lambda_{j}\in\mathrm{Spec}\,\Delta}}\frac{\phi_{\lambda_{j}}(x_{1})^{2}}{\lambda_{j}^{s}}
\]
can be continued as a meromorphic function of $s$ with a unique simple
pole at $N/2$ and negative integer zeros. Next, the analytic continuation
of the spectral zeta function was studied by Minakshisundaram and
Pleijel \cite{MP:1949} for the Laplacian on a general compact manifold
by a method that is a generalization of Carleman's. Seeley \cite{See:1968}
studied the spectral zeta function of an elliptic $\psi$do on a compact
manifold without boundary through the trace of complex powers of $\psi$dos,
furthermore giving the value of the zeta function at $0$. 

Many different techniques have been used to obtain properties of the
spectral zeta function. Duistermaat and Guillemin \cite{Du-Gu:1975}
(see also, \cite{G:1994} and the references in Hormander \cite{Ho:1968})
studied systematically the spectral zeta function of $\psi$dos on
compact bounderyless manifolds basing their approach on the construction
of a parametrix for the wave equation. Robert \cite{R1:1978} (see
also Aramaki \cite{A:1989}) extended meromorphically the spectral
zeta function of an elliptic $\psi$do on $\mathbb{R}^{n}$ to the
whole complex plane with simple poles that he computed along with
the corresponding residues. He generalized to the global setting the
techniques by Seeley to construct the parametrix of the resolvent
by complex powers.

An import distinction to show the relevance of the results in this
paper is the one between \textit{regular} and \textit{semiregular}
symbols. Since the natural homogeneity of the Poisson bracket of homogeneous
symbols is the sum of the orders minus $2$, it is natural in the
global calculus to call \textquotedblleft regular\textquotedblright{}
those symbols whose asymptotic expansion is made of homogeneous symbols
for which the $j$-th term has order $\mu-2j$ where $\mu$ is the
order of the principal term. We will call \textquotedblleft semiregular\textquotedblright{}
those symbol whose $j$-th term in the asymptotic expansion has order
$\mu-j$. This is indeed parallel to the use of \textquotedblleft semiregular\textquotedblright{}
appearing in the paper by Boutet De Monvel \cite{BdM:1974} on the
hypoellipticity of the $\overline{\partial}$ operator. For a semiregular
system of order $\mu$ we will call \textit{semiprincipal symbol}
the term of degree $\mu-1$ in the asymptotic expansion of its symbol
while we call \textit{principal symbol} the one of order $\mu$.

Moreover, following the discussion by \cite{P1:2010}, \cite{P2:2006},
\cite{PW1:2001}, and \cite{PW2:2003}, we call second order regular
Non-Commutative Harmonic Oscillators (NCHOs) the class of the \textit{regular}
global partial differential systems of second order with polynomial
coefficients. From now on we will omit the expression ``second order''
since all the NCHOs considered will be of second order.

Ichinose and Wakayama \cite{IW:2007} obtained a meromorphic continuation
of the spectral zeta function of a subclass of regular NCHOs and determined
some of its special values. In addition, they showed that such a spectral
zeta function has only a simple pole at $1$ and that the sequence
of its trivial zeros coincides with the one of the Riemann zeta function,
the non-positive even integers. Their approach is based on the Mellin
transform of the heat-semigroup of the operator in the approximation
given by a parametrix which they computed \textit{directly}, without
using the one for the resolvent, obtaining its asymptotic expansion
(see (15) and (16) in their paper). Later, Parmeggiani \cite{P1:2010}
generalized that approach to obtain the meromorphic continuation of
the spectral zeta function of all the regular NCHOs. Nevertheless,
while gaining in generality unfortunately his result did not explicitly
locate the trivial zeros of the continuation of the spectral zeta
function as could Ichinose and Wakayama.

Ichinose and Wakayama's and Parmeggiani's papers deal with \textit{regular}
systems. Regarding the \textit{semiregular} systems, Sugiyama explored
in \cite{Su:2018} the Hurwitz-type spectral zeta function for the
quantum Rabi model (describing the interaction of light and matter
of a two-level atom coupled with a single quantized photon of the
electromagnetic field, see the seminal papers \cite{Ra1:1936} and
\cite{Ra2:1937} by Rabi, see also \cite{Br:2014} by Braak).

In this paper we study the properties of the spectral zeta function
associated with a positive elliptic \textit{semiregular positive partial
differential systems} with polynomial coefficients, including \textit{also}
models of semiregular NCHOs in the class of the Semiregular Metric
Globally Elliptic Systems (SMGES) as those introduced in Section 3
of Malagutti and Parmeggiani \cite{MaPa:2022}. The class of the SMGES
is given by those matrix-valued symbols with a scalar principal part
(that is the ``metric'' part) and a smoothly diagonalizable semiprincipal
part. This class contains models relevant to Quantum Optics, such
as the Jaynes-Cummings model (which describes the interaction between
an atom and the electromagnetic field in a cavity and can be derived
as an approximation of Rabi's by rotating waves approximation provided
that the coupling strength is sufficiently weak, see \cite{SK:1963}
and the seminal paper \cite{JC:1963} by Jaynes and Cummings). Here
we follow the construction of the zeta function provided by Ichinose
and Wakayama, in analogy to the approach by Parmeggiani in Theorem
7.2.1 of \cite{P1:2010}.

We will prove a result about the continuation of the spectral zeta
function $\zeta_{A^{{\rm w}}}$ which turns out to be a meromorphic
function whose poles are real and accumulate at $-\infty$. Namely,
we will give the continuation as a linear combination of the meromorphic
functions $s\mapsto\frac{1}{s-(n-j)+h/2}$ , $j\geq0$ and $h=0$, $1$,
modulo a function that is holomorphic on a complex half-plan. Notice
that indeed our extension can have poles in all the negative semi-integers,
unlike the results in \cite{IW:2007}, \cite{P1:2010} and \cite{Su:2018}
where the poles are all positive. The meromorphic continuation is
obtained by following the approach of Theorem 7.2.1 in \cite{P1:2010}
where the  parametrix approximation $U_{A}(t)$ of the heat-semigroup
$e^{-tA^{{\rm w}}}$is used. More precisely, by the Mellin transform
we can write $\zeta_{A^{{\rm w}}}$ as $s\mapsto\frac{1}{\Gamma(s)}\int_{0}^{+\infty}t^{s-1}\mathrm{Tr}\:e^{-tA^{{\rm w}}}\:dt$
for $\mathrm{Re}s>2n/2=n$ and, at this point, the asymptotic expansion
${\displaystyle \sum_{j\geq0}}b_{-j}(t)$ (in the sense of Remark
6.1.5 at p. 83 of \cite{P1:2010}) of $U_{A}(t)$ with $t\in\overline{\mathbb{R}}_{+}$
becomes crucial. In fact, the approximation of $s\mapsto\frac{1}{\Gamma(s)}\int_{0}^{+\infty}t^{s-1}\mathrm{Tr}\:e^{-tA^{{\rm w}}}\:dt$
by $s\mapsto\frac{1}{\Gamma(s)}\int_{0}^{+\infty}t^{s-1}\mathrm{Tr}\:U_{A}(t)\:dt$
leads to the study of integrals of the form 
\begin{equation}
(2\pi)^{-n}\int_{\mathbb{R}^{2n}}\chi(X)\mathsf{Tr}\left(b_{-2j-h}(t,X)\right)\,dX,\;j\in\mathbb{N},\,h=0,1,\label{eq:Introd-coeff}
\end{equation}
where $\chi$ is a chosen excision function and $\mathsf{Tr}$ is
the classical matrix trace. In fact the computation of (\ref{eq:Introd-coeff})
will give the coefficients of the linear combination of the aforementioned
meromorphic functions. These coefficients will contribute to determine
the residues and zeros of the spectral zeta function. Now one needs
to go through a Taylor expansion argument as the time variable $t\rightarrow0+$
of the terms arising from the study of $\mathrm{Tr}\:e^{-tA^{{\rm w}}}-{\displaystyle \sum_{j=0}^{\nu}}{\displaystyle \,\sum_{h=0}^{1}}\mathrm{Tr}\:B_{-2j-h}(t)$
(where $B_{-k}$ with principal symbol $b_{-k}$).
(The behavior of $e^{-tA^{{\rm w}}}$ as $t\rightarrow+\infty$ does
not affect the result.)

This is a delicate argument since the behaviour of the coefficients
of the linear combination of the above meromorphic functions must
be controlled as $t\rightarrow0+$.

The plan of the paper is the following. First of all, the notation
adopted will be introduced in Section \ref{sec:Parabolic-calculus}
along with the parabolic $\psi$differential calculus needed to define
the heat-semigroup parametrix which will be constructed directly in
Section \ref{sec:Heat-semigroup-parametrix} by computing the terms
of its asymptotic expansion through the solution of eikonal and transport
equations. After that, in Section \ref{sec:Time-decrease-property},
we will control the behaviour of the coefficients. We will give the
proof of our theorem in Section \ref{sec:Meromorphic-continuation-of-zeta}.
Actually, in Section \ref{sec:Meromorphic-continuation-of-zeta} we
will also obtain a meromorphic continuation for the Hurwitz spectral
zeta function $\zeta_{A^{\mathrm{w}}+\tau I}$ for all $\tau\geq0$.
Finally, in Section \ref{sec:Example} by using our results in this
paper we will compute the meromorphic continuation of the spectral
zeta function for the Hamiltonians of Jaynes-Cummings and its generalization
to a $3$-level atom in one cavity. For these Hamiltonians we will
show that the meromorphic continuation has only a simple pole at $s=1$
and no other (even if, recall, the general formula allows all the
negative semi-integer as poles).

\section{\label{sec:Parabolic-calculus}Parabolic calculus}

In this section, similarly to what is done by Parenti and Parmeggiani
in \cite{PP:2004} (see also Section 6.1 of \cite{P1:2010}), we will
introduce a class of symbols suitable for the construction of a pseudodifferential
approximation of $e^{-tA^{\mathrm{w}}}$. Let us recall $\overline{\mathbb{R}}_{+}=[0,+\infty)$.

\begin{definition}

Let $r\in\mathbb{R}$. By $S(\mu,r)$ we denote the set of all smooth
maps $b:\,\overline{\mathbb{R}}_{+}\times\mathbb{R}^{n}\times\mathbb{R}^{n}\longrightarrow M_{N}$
satisfying the following estimates: for any given $\alpha\in\mathbb{Z}_{+}^{2n}$
and any given $p$, $j\in$$\mathbb{Z}_{+}$there exists $C>0$ such
that
\begin{equation}
\sup\left|t^{p}(\frac{d}{dt})^{j}\partial_{X}^{\alpha}b(t,X)\right|\leq cm(X)^{r-|a|+(j-p)\mu}.
\end{equation}
For $b\in S(\mu,r)$ we then consider the pseudodifferential operator
\[
b^{\mathrm{w}}(t,x,D)u(x)=(2\pi)^{-n}\iint e^{i(x-y,\xi)}b(t,\frac{x+y}{2},\xi)u(y)dyd\xi,\,u\in\mathscr{S}(\mathbb{R}^{n},\mathbb{C}^{N}),
\]
 and we shall say that $B\in\mathrm{OP}S(\mu,r)$ if $B=b^{\mathrm{w}}(t,x,D)+R$,
where $R$ is smoothing. In this setting, a smoothing operator $R$
is any \textbf{continuous} map
\[
R:\,\mathscr{S}'(\mathbb{R}^{n};\mathbb{C}^{N})\longrightarrow\mathscr{S}(\overline{\mathbb{R}}_{+};\mathscr{S}(\mathbb{R}^{n};\mathbb{C}^{N})).
\]

\end{definition}

Then we introduce the ``classical operators\textquotedblright : in
this case the key is to take in account the correct homogeneity properties.
The basic example to keep in mind is the matrix $e^{-ta_{\mu}(x,\xi)}$
. 

\begin{definition}

We say that the operator $B\in\mathrm{OP}S(\mu,r)$, $B=b^{\mathrm{w}}+R$
is \textbf{classical}, and write $B\in\mathrm{OP}S_{\mathrm{cl}}(\mu,r)$
, if there exists a sequence of functions $b_{r-2j}=b_{r-2j}(t,X),$$\,j\geq0,t\geq0$
and $X\neq0$, such that:
\begin{enumerate}
\item One has the homogeneity
\begin{equation}
b_{r-2j}(t,\tau X)=\tau^{r-2j}b_{r-2j}(\tau^{\mu}t,X),\:\forall\tau>0,\,\forall j\geq0;\label{eq:Homogeneity}
\end{equation}
\item The function
\[
\mathbb{R}^{2n}\setminus\{0\}\ni X\longmapsto b_{r-2j}(\cdot,X)\in\mathscr{S}(\overline{\mathbb{R}}_{+},M_{N}),
\]
is smooth for all $j\geq0$;
\item For all $\nu\geq1$
\begin{equation}
b(t,X)-\sum_{j=0}^{\nu=1}\chi(X)b_{r-2j}(t,X)\in S(\mu,r-2\nu),\label{eq:AsympototicExpansion}
\end{equation}
 where $\chi$ is an excision function.
\end{enumerate}
\end{definition}

\begin{remark}

We call $b_{r}=\sigma_{r}(B)$ the \textbf{principal symbol} of $B$.

\end{remark}

\begin{remark}

\textbf{Semi-regular classical }symbols are defined accordingly, considering
also terms with odd degree of homogeneity in the expansion formula
(\ref{eq:AsympototicExpansion}), and the class of pseudodifferential
operators associated to them is denoted by $\mathrm{OP}S_{\mathrm{sreg}}(\mu,r)$.

\end{remark}

\section{\label{sec:Heat-semigroup-parametrix}Parametrix of the heat-semigroup}

In this section we will construct the parametrix of the heat-semigroup
of a semiregular positive elliptic pseudodifferential operator.

\begin{lemma} \label{lem:Heat-Parametrix}

Let $A=A^{*}$, with $A\sim\sum_{j\geq0}a_{2-j}\in S_{\mathrm{sreg}}(m^{2},g;\mathsf{M}_{N}),$
be an elliptic second order system such that $A^{{\rm w}}>0$. Then,
there exists $U_{A}\in{\rm OP}S_{\mathrm{sreg}}(\mu,0)$ such that
\[
\frac{d}{dt}U_{A}+A^{{\rm w}}U_{A}:\mathscr{S}^{'}(\mathbb{R}^{n};\mathbb{C}^{N})\rightarrow\mathscr{S}(\overline{\mathbb{R}}_{+};\mathscr{S}(\mathbb{R}^{n};\mathbb{C}^{N}))
\]
 is smoothing, and
\[
U_{A}|_{t=0}-I_{N}:\mathscr{S}^{'}(\mathbb{R}^{n};\mathbb{C}^{N})\rightarrow\mathscr{S}(\mathbb{R}^{n};\mathbb{C}^{N})
\]
 is smoothing. Moreover, the principal symbol of $U_{A}$ is 
\[
\overline{\mathbb{R}}_{+}\times\left(\mathbb{R}^{2n}\setminus\{0\}\right)\ni(t,X)\mapsto e^{-ta_{\mu}(X)}.
\]

\end{lemma}

\begin{proof}

We will prove the lemma by constructing the terms of the expansion
of the symbol of $U_{A}$. In fact, we determine those terms by solving
a sequence of transport equations.

Let 
\[
\overline{\mathbb{R}}_{+}\times\left(\mathbb{R}^{2n}\setminus\{0\}\right)\ni(t,X)\mapsto b_{0}(t,X):=e^{-ta_{\mu}(X)},
\]
and let $B_{0}\in{\rm OP}S_{{\rm sreg}}(\mu,0)$ with principal symbol
given by $b_{0}$. Hence, by Lemma 6.1.3 at p. 81 of \cite{P1:2010}
we have that $\frac{d}{dt}B_{0}+A^{{\rm w}}B_{0}\in{\rm OP}S_{{\rm sreg}}(\mu,\mu-1)$
with principal symbol $r_{\mu-1}:=a_{\mu-1}b_{0}$. Moreover, $B_{0}|_{t=0}-I_{N}$
is a pseudodifferential system with symbol in $S_{\mathrm{sreg}}(m^{-1},g;\mathsf{M}_{N})$
and we denote its principal symbol by $p_{-1}$.

Next, we look for a symbol $b_{-1}(t,X)$, positively homogeneous
of degree $-1$ (in the sense of (\ref{eq:Homogeneity})), such that
\begin{equation}
\begin{cases}
\frac{d}{dt}b_{-1}+a_{\mu}b_{-1}=-r_{\mu-1},\\
b_{-1}|_{t=0}=-p_{-1}.
\end{cases}\label{eq:Syst_b_-1}
\end{equation}

The solution of (\ref{eq:Syst_b_-1}),
\[
b_{-1}(t,X):=-e^{-ta_{\mu}(X)}p_{-1}(X)-\int_{0}^{t}e^{-(t-t^{'})a_{\mu}(X)}r_{\mu-1}(t^{'},X)\,dt^{'},
\]
 is easily seen to be smooth and have the required homogeneity properties
since 
\begin{align*}
b_{-1}(t,\tau X)= & -e^{-ta_{\mu}(X)}p_{-1}(X)-\int_{0}^{t}e^{-(t-t^{'})a_{\mu}(X)}r_{\mu-1}(t^{'},X)\,dt^{'}\\
= & -e^{-\tau^{\mu}ta_{\mu}(X)}\tau^{-1}p_{-1}(X)-\int_{0}^{t}e^{-\tau^{\mu}(t-t^{'})a_{\mu}(X)}\tau^{\mu-1}r_{\mu-1}(\tau^{\mu}t^{'},X)\,dt^{'}\\
= & \tau^{-1}\left(-e^{-\tau^{\mu}ta_{\mu}(X)}p_{-1}(X)-\int_{0}^{\tau^{\mu}t}e^{-(\tau^{\mu}t-t^{'})a_{\mu}(X)}r_{\mu-1}(t^{'},X)\,dt^{'}\right)\\
= & \tau^{-1}b_{-1}(\tau^{\mu}t,X),
\end{align*}
where the last equality follows from the change of variable $t\rightarrow\tau^{-\mu}t$
in the integral. Taking $B_{-1}\in\mathrm{OP}S_{{\rm sreg}}(\mu,-1)$
with principal symbol given by $b_{-1}$ gives
\[
\frac{d}{dt}(B_{0}+B_{-1})+A^{{\rm w}}(B_{0}+B_{-1})\in{\rm OP}S_{\mathrm{sreg}}(\mu,\mu-2).
\]

Moreover, $(B_{0}+B_{-1})|_{t=0}-I_{N}$ is a pseudodifferential system
with symbol in $S_{\mathrm{sreg}}(m^{-2},g;\mathsf{M}_{N})$ and we
denote its principal symbol by $p_{-2}$.

Iterating the above procedure gives a formal series 
\[
\sum_{k\geq0}B_{-k},\;B_{-k}\in\mathrm{OP}S_{\mathrm{sreg}}(\mu,-k).
\]

Hence, there exist an operator $U_{A}\in\mathrm{OP}S_{\mathrm{sreg}}(\mu,0)$
for which
\[
U_{A}-\sum_{k=0}^{\nu-1}B_{-k}\in\mathrm{OP}S(\mu,-\nu),\,\forall\nu\geq1,
\]
 by an adaptation of Proposition 3.2.15 at p. 32 of \cite{P1:2010}
and therefore we obtain the required parametrix.

\end{proof}

\begin{remark} \label{rem:ConstructB_-j}

In the applications of Lemma \ref{lem:Heat-Parametrix} we shall always
consider a parametrix approximation of $e^{-tA^{\mathrm{w}}}$ where
$b_{-j}|_{t=0}=0$ for $j\geq1$,
\[
B_{-j}:=(\chi b_{-j})^{\mathrm{w}}(t,x,D_{x}),
\]
 for all $t\in\overline{\mathbb{R}}_{+}$, where $\chi$ is a chosen
excision function. Hence, consider the symbol $c_{A}(t,X)$ of $U_{A}(t)$,
i.e. $U_{A}(t)=c_{A}^{\mathrm{w}}(t,x,D)$, given by
\begin{equation}
c_{A}(t,X)=\sum_{j\geq0}\chi_{j}(X)b_{-j}(t,X),\label{eq:c_A_Series}
\end{equation}
where $\chi_{0}(X):=\chi(X)$ and $\chi_{j}(X):=\chi(X/R_{j})$, $j\geq1$,
with $R_{j}\nearrow+\infty$, as $j\rightarrow+\infty$, sufficiently
fast (for instance, see the proof of Proposition 3.2.15 at p. 32 of
\cite{P1:2010}). Thus, the series (\ref{eq:c_A_Series}) is locally
finite in $X$ and, hence, $c_{A}(t,\cdot)\in C^{\infty}$ for all
$t\in\overline{\mathbb{R}}_{+}$.

From now on we will write $U_{A}\sim{\displaystyle \sum_{j\geq0}}B_{-j}$
.

\end{remark}

\section{\label{sec:Time-decrease-property}Vanishing property}

Let $A^{\mathrm{w}}$ be as in the previous section. In this section
we prove the technical proposition that we need to control the behavior
of the $b_{-j}$ constructed in Lemma \ref{lem:Heat-Parametrix} as
$t\rightarrow0+$, that is, its vanishing property, for a class of
positive and self-adjoint elliptic \textit{differential} systems with
symbol in $S_{\mathrm{sreg}}(m^{2},g;\mathsf{M}_{N})$. Hence, we
will suppose the symbol of $A^{\mathrm{w}}$ to be $a_{2}+a_{1}+a_{0}$
where $a_{j}$ is an $N\times N$ matrix-valued function on $\mathbb{R}^{2n}$
with homogeneous polynomial of degree $j$ entries for all $j=0,$
$1$, $2$.

\begin{proposition} \label{prop:TimeDecrease-Diff}

Let $A=a_{2}+a_{1}+a_{0}$ be an elliptic of second order where $a_{j}$
is an $N\times N$  matrix-valued function on $\mathbb{R}^{2n}$ with
homogeneous polynomial of degree $j$ entries for all $j=0,1,2$,
let $A^{{\rm w}}>0$ , and let $U_{A}$ be the heat-semigroup $e^{-tA^{{\rm w}}}$
parametrix constructed by Lemma \ref{lem:Heat-Parametrix}. Then,
denoting again by ${\displaystyle \sum_{j\geq0}}B_{-j}$ the expansion
of $U_{A}$ constructed in the proof of Lemma \ref{lem:Heat-Parametrix}
and by $b_{-j}$ the principal symbol of $B_{-j}$, we have for all
$j\geq0$ and $h=0$, $1$
\[
b_{-2j-h}(t,\omega)=O(t^{j+h}),\;t\rightarrow0+,
\]
 and for all $\alpha$, $\beta\in\mathbb{Z}_{+}^{n}$, with $|\alpha|=2k+1$,
$k\geq0$ and $|\beta|\leq1$ we have:
\[
\partial_{X}^{\alpha+\beta}b_{-2j-h}(t,\omega)=O(t^{j+k+h|\beta|+1}),\;t\rightarrow0+,
\]
 where the constants in $O(\cdot)$ do not depend on $\omega\in\mathbb{S}^{2n-1}$. 

\end{proposition}

\begin{proof}

We prove this theorem by induction taking into account the definition
of the terms $b_{-j}$ and making straightforward computations.

We won\textquoteright t be writing the dependence on $\omega$, and
we will write $b_{-j}^{(\ell)}$ for a generic $\partial_{X}^{\alpha}b_{-j}$
with $|\alpha|=\ell$.

First of all, we remind that given two pseudodifferential operators
with symbol $a$ and $b$, then by the composition law for pseudodifferential
operators (see, for instance, formula (3.3) at p. 19 of \cite{P1:2010})
$a^{{\rm w}}b^{{\rm w}}$ has symbol
\[
a\#b\sim ab+\sum_{j\geq1}\frac{1}{j!}\left(\frac{-i}{2}\right)^{j}\{a,b\}_{(j)},
\]
 where $\{\cdot,\cdot\}_{(1)}=\{\cdot,\cdot\}$ is the Poisson bracket.

The terms $r_{2-j}$, $j\geq1$ obtained in the proof of Lemma \ref{lem:Heat-Parametrix}
is 
\begin{align}
r_{2-j}= & a_{0}b_{-(j-2)}+a_{1}b_{-(j-1)}+\frac{1}{2}\left(\frac{-i}{2}\right)^{2}\{a_{2},b_{-(j-4)}\}_{(2)}\label{eq:r-Definition-Diff}\\
 & -\frac{i}{2}\{a_{2},b_{-(j-2)}\}-\frac{i}{2}\{a_{1},b_{-(j-3)}\},\,\,j\geq0,\nonumber 
\end{align}
where we set $b_{k}\equiv0$ for all $k=1,\ldots,4$ and we recall
that $a_{0}$ is a \textit{constant} $N\times N$ Hermitian matrix.
Therefore, by the construction in the proof of Lemma \ref{lem:Heat-Parametrix},
\begin{equation}
\begin{cases}
b_{0}(t,X)=e^{-ta_{2}(X)},\\
b_{-j}(t,X)=-\int_{0}^{t}e^{-(t-t^{'})a_{2}}r_{2-j}(t^{'},X)\,dt^{'},\:j\geq1.
\end{cases}\label{eq:b-Formula-Diff}
\end{equation}
In fact, $p_{-j}=0$ for any $j\geq1$ under our hypotheses. 

Denote by $E(a_{2}^{(2)},b_{-j}^{(2)})$, resp. $E(a_{2}^{(1)},b_{-j}^{(1)})$,
a generic expression obtained by taking the (matrix) product of derivatives
of order $2$, resp. order $1$, of $a_{2}$ with derivatives of order
$2$, resp. order $1$, of $b_{-j}$. Hence for all $j\geq0$,
\[
\{a_{2},b_{-j}\}=E(a_{2}^{(1)},b_{-j}^{(1)}),\,\text{ and }\,\{a_{2},b_{-j}\}_{(2)}=E(a_{2}^{(2)},b_{-j}^{(2)}).
\]

Therefore, in $\{a_{2},b_{-j}\}_{(2)}=E(a_{2}^{(2)},b_{-j}^{(2)})$
we have a \textit{constant} coefficient matrix (given by partial derivatives
of order $2$ of $a_{2}$) times partial derivatives of order $2$
of $b_{-j}$. 

We proceed by induction. We start with the case $j=0$ and $h=0$.
In this case $b_{0}$ is the solution of
\begin{equation}
\begin{cases}
\partial_{t}b_{0}+a_{2}b_{0}=0,\\
b_{0}|_{t=0}=I_{N},
\end{cases}\label{eq:Syst-b_0-Diff}
\end{equation}
 whence $b_{0}(t)=O(1)$ as $t\rightarrow0+$.

Next, by induction on $\ell$ we show that $b_{0}^{(\ell)}$ has the
claimed property.

For $\ell=1$ we take a $1$st-order partial derivative with respect
to $X$ of (\ref{eq:Syst-b_0-Diff}) and have
\[
\begin{cases}
\partial_{t}b_{0}^{(1)}+a_{2}b_{0}^{(1)}=-a_{2}^{(1)}b_{0},\\
b_{0}^{(1)}|_{t=0}=0,
\end{cases}
\]
 whence
\begin{equation}
b_{0}^{(1)}(t)=-\int_{0}^{t}e^{-(t-t^{'})a_{2}}a_{2}^{(1)}b_{0}(t^{'})\,dt^{'}=O(t),\;t\rightarrow0+.\label{eq:b^(1)_0-Explicit}
\end{equation}

For $\ell=2$ we take a $1$st-order partial derivative with respect
to $X$ of (\ref{eq:b^(1)_0-Explicit}) and have
\begin{align*}
b_{0}^{(2)}(t)= & -\int_{0}^{t}(e^{-(t-t^{'})a_{2}})^{(1)}a_{2}^{(1)}b_{0}(t^{'})\,dt^{'}-\int_{0}^{t}e^{-(t-t^{'})a_{2}}a_{2}^{(2)}b_{0}(t^{'})\,dt^{'}\\
 & -\int_{0}^{t}e^{-(t-t^{'})a_{2}}a_{2}^{(1)}b_{0}(t^{'})^{(1)}\:dt^{'}\\
= & O(t^{2})+O(t)+O(t^{2})=O(t),\;t\rightarrow0+.
\end{align*}

Next, suppose $b_{0}^{(2k-1+\ell)}(t)=O(t^{k})$ as $t\rightarrow0+$,
for $\ell=0,1$ and $k\geq0$. We want to prove that $b_{0}^{(2k+1+\ell)}(t)=O(t^{k+1})$,
as $t\rightarrow0+$, for $\ell=0,1$. Using (\ref{eq:Syst-b_0-Diff})
and taking a $2k+1$-st partial derivative with respect to $X$ we
obtain (recall that $a_{2}^{(p)}=0$ for all $p\geq3$ since $a_{2}$
has polynomial of degree $2$ entries)
\[
\begin{cases}
\partial_{t}b_{0}^{(2k+1)}+a_{2}b_{0}^{(2k+1)}=-a_{2}^{(1)}b_{0}^{(2k)}-a_{2}^{(2)}b_{0}^{(2k-1)}=O(t^{k})+O(t^{k})=O(t^{k}),\\
b_{0}^{(2k+1)}|_{t=0}=0,
\end{cases}
\]
whence $b_{0}^{(2k+1)}(t)=O(t^{k+1})$ as $t\rightarrow0+$. Then,
as before,
\begin{align*}
b_{0}^{(2k+2)}(t)= & -\int_{0}^{t}(e^{-(t-t^{'})a_{2}})^{(1)}(a_{2}^{(1)}b_{0}^{(2k)}(t^{'}))+a_{2}^{(2)}b_{0}^{(2k-1)}(t^{'}))\,dt^{'}\\
 & -\int_{0}^{t}e^{-(t-t^{'})a_{2}}\partial_{X}(a_{2}^{(1)}b_{0}^{(2k)}(t^{'}))+a_{2}^{(2)}b_{0}^{(2k-1)}(t^{'}))\:dt^{'}\\
= & O(t^{k+2})+O(t^{k+1})=O(t^{k+1}),\;t\rightarrow0+.
\end{align*}
Hence, the result is proved for $b_{0}$.

Next, we prove the result for the case $j=0$ and $h=1$. In this
case by (\ref{eq:b-Formula-Diff})
\begin{align}
b_{-1}(t)= & -\int_{0}^{t}e^{-(t-t^{'})a_{2}}r_{2-1}(t^{'})\,dt^{'}\nonumber \\
= & -\int_{0}^{t}e^{-(t-t^{'})a_{2}}a_{1}\underbrace{b_{0}(t^{'})}_{=O(1),\,t^{'}\rightarrow0+}\,dt^{'}\label{eq:b_-1Formula}\\
= & O(t),\;t\rightarrow0+.\nonumber 
\end{align}
By taking the derivative in $X$ of (\ref{eq:b_-1Formula})
\begin{align*}
b_{-1}^{(1)}(t)= & -\int_{0}^{t}(e^{-(t-t^{'})a_{2}})^{(1)}a_{1}b_{0}(t^{'})\,dt^{'}\\
 & -\int_{0}^{t}e^{-(t-t^{'})a_{2}}a_{1}^{(1)}b_{0}(t^{'})\,dt^{'}\\
 & -\int_{0}^{t}e^{-(t-t^{'})a_{2}}a_{1}\underbrace{b_{0}^{(1)}(t^{'})}_{=O(t^{'}),\,t^{'}\rightarrow0+}\,dt^{'}\\
= & O(t^{2})+O(t)+O(t^{2})=O(t),\:t\rightarrow0+,
\end{align*}
and by taking another derivative in $X$ we obtain that $b_{-1}^{(2)}(t)=O(t^{2})$
(recall that $a_{1}^{(p)}=0$ for all $p\geq2$ since $a_{1}$ has
polynomial of degree $1$ entries). 

Next, suppose $b_{-1}^{(2k-1+\ell)}(t)=O(t^{k+\ell})$, as $t\rightarrow0+$,
for $\ell=0,1$ and $k\geq0$. We want to prove that $b_{-1}^{(2k+1+\ell)}(t)=O(t^{k+\ell+1})$,
as $t\rightarrow0+$, for $\ell=0$, $1$. First of all, we notice
that, by (\ref{eq:b-Formula-Diff}), $b_{-1}$ is the solution of
the Cauchy problem 
\begin{equation}
\begin{cases}
\partial_{t}b_{-1}+a_{2}b_{-1}=-r_{1}=-a_{1}b_{0},\\
b_{-1}|_{t=0}=0,
\end{cases}\label{eq:Syst-b_-1}
\end{equation}
By using (\ref{eq:Syst-b_-1}) and taking a $2k+1$-st partial derivative
with respect to $X$
\[
\begin{cases}
\begin{array}{rl}
\partial_{t}b_{-1}^{(2k+1)}+a_{2}b_{-1}^{(2k+1)}= & -a_{2}^{(1)}b_{-1}^{(2k)}-a_{2}^{(2)}b_{-1}^{(2k-1)}-a_{1}^{(1)}b_{0}^{(2k)},\\
= & O(t^{k+1})+O(t^{k})+O(t^{k})
\end{array}\\
b_{-1}^{(2k+1)}|_{t=0}=0,
\end{cases}
\]
 whence $b_{-1}^{(2k+1)}(t)=O(t^{k+1})$ as $t\rightarrow0+$. Then,
as before,
\begin{align*}
b_{-1}^{(2k+2)}(t)= & -\int_{0}^{t}(e^{-(t-t^{'})a_{2}})^{(1)}(a_{2}^{(1)}b_{-1}^{(2k)}(t^{'})+a_{2}^{(2)}b_{-1}^{(2k-1)}(t^{'})+a_{1}^{(1)}b_{0}^{(2k)}(t^{'}))\,dt^{'}\\
 & -\int_{0}^{t}e^{-(t-t^{'})a_{2}}\partial_{X}(a_{2}^{(1)}b_{-1}^{(2k)}(t^{'})+a_{2}^{(2)}b_{-1}^{(2k-1)}(t^{'})+a_{1}^{(1)}b_{0}^{(2k)}(t^{'}))\:dt^{'}\\
= & O(t^{k+2})+O(t^{k+2})=O(t^{k+2}),\;t\rightarrow0+.
\end{align*}
Hence, the result has been proved for $b_{-1}$.

Next, suppose, by induction, that for all $\ell=0,1$, all $h=0,1$
and all $j^{'}\leq j$ 
\[
b_{-2j^{'}-h}=O(t^{j^{'}+h}),\,\,b_{-2j^{'}-h}^{(2k+1+\ell)}=O(t^{j^{'}+k+h\ell+1}),\quad t\rightarrow0+.
\]
We want to prove $b_{-2(j+1)}=O(t^{j+1})$ and $b_{-2(j+1)}^{(2k+1+\ell)}=O(t^{j+1+k+1})$
for $\ell=0,1$, as $t\rightarrow0+$, that is, the case $h=0$ (after
that, we will prove that $b_{-2(j+1)-1}=O(t^{j+1+1})$ and $b_{-2(j+1)-1}^{(2k+1+\ell)}=O(t^{j+1+k+\ell+1})$
for $t\rightarrow0+$, i.e. the case $h=1$). To do it, we have to
examine $r_{2-2(j+1)}$ (see (\ref{eq:b-Formula-Diff})). In the first
place we have from (\ref{eq:r-Definition-Diff})
\begin{align*}
\hspace{-0.8cm}r_{2-2(j+1)}= & a_{0}b_{-2j}+a_{1}b_{-2j-1}+\frac{1}{2}\left(\frac{-i}{2}\right)^{2}\{a_{2},b_{-2(j-1)}\}_{(2)}-\frac{i}{2}\{a_{2},b_{-2j}\}-\frac{i}{2}\{a_{1},b_{-(2j-1)}\}\\
= & O(t^{j})+O(t^{j+1})+O(t^{j-1+1})+O(t^{j+1})+O(t^{j-1+1})\\
= & O(t^{j}),\quad t\rightarrow0+.
\end{align*}
Consider next, keeping into account that $a_{q}^{(p)}=0$ for all
$p\geq q+1$ since $a_{q}$ has polynomial of degree $q=1$, $2$
entries,
\begin{align*}
r_{2-2(j+1)}^{(2k+1)}= & a_{0}b_{-2j}^{(2k+1)}+a_{1}b_{-2j-1}^{(2k+1)}+E(a_{1}^{(1)},b_{-2j-1}^{(2k)})+E(a_{2}^{(2)},b_{-2(j-1)}^{(2k+3)})+E(a_{2}^{(1)},b_{-2j}^{(2k+2)})\\
 & +E(a_{2}^{(2)},b_{-2j}^{(2k+1)})+E(a_{1}^{(1)},b_{-(2j-1)}^{(2k+2)})\\
= & O(t^{j+k+1})+O(t^{j+k+1})+O(t^{j+k-1+1+1})+O(t^{j-1+k+1+1})+O(t^{j+k+1})\\
 & +O(t^{j+k+1})+O(t^{j-1+k+1+1})\\
= & O(t^{j+k+1}),\quad t\rightarrow0+.
\end{align*}
Taking an extra derivative, one immediately sees also that
\[
r_{2-2(j+1)}^{(2k+2)}=O(t^{j+k+1}),\quad t\rightarrow0+.
\]
Hence, for all $\ell=0,1$ and for $k\geq-1$
\[
r_{2-2(j+1)}^{(2k+1+\ell)}=O(t^{j+k+1}),\quad t\rightarrow0+
\]
 (when $k=-1$ we take $\ell=1$). Since $b_{-2(j+1)}$ is the solution
of the Cauchy problem
\begin{equation}
\begin{cases}
\partial_{t}b_{-2(j+1)}+a_{2}b_{-2(j+1)}=-r_{2-2(j+1)},\\
b_{-2(j+1)}|_{t=0}=0,
\end{cases}\label{eq:Syst-b_-2(j+1)}
\end{equation}
we obtain $b_{-2(j+1)}(t)=O(t^{j+1})$ as $t\rightarrow0+$. As before,
taking one partial derivative with respect to $X$ yields
\[
\begin{cases}
\partial_{t}b_{-2(j+1)}^{(1)}+a_{2}b_{-2(j+1)}^{(1)}=-a_{2}^{(1)}b_{-2(j+1)}-r_{2-2(j+1)}^{(1)}=O(t^{j+1})+O(t^{j+1}),\\
b_{-2(j+1)}^{(1)}|_{t=0}=0,
\end{cases}
\]
whence it follows that $b_{-2(j+1)}^{(1)}(t)=O(t^{j+2})$, and, taking
an extra derivative, also that, as $t\rightarrow0+$,
\[
b_{-2(j+1)}^{(2)}(t)=-\partial_{X}\left(\int_{0}^{t}e^{-(t-t^{'})a_{2}}\left(a_{2}^{(1)}b_{-2(j+1)}+r_{2-2(j+1)}^{(1)}\right)dt^{'}\right)=O(t^{j+2}).
\]
Supposing then by induction the estimates up to order $2k-1$ and
using
\[
\begin{cases}
\begin{array}{rl}
\partial_{t}b_{-2(j+1)}^{(2k+1)}+a_{2}b_{-2(j+1)}^{(2k+1)}= & -E(a_{2}^{(1)},b_{-2(j+1)}^{(2k)})-E(a_{2}^{(2)},b_{-2(j+1)}^{(2k-1)})-r_{2-2(j+1)}^{(2k+1)},\\
= & O(t^{j+1+k-1+1})+O(t^{j+1+k-1+1})+O(t^{j+k+1}),
\end{array}\\
b_{-1}^{(2k+1)}|_{t=0}=0,
\end{cases}
\]
 we obtain $b_{-2(j+1)}^{(2k+1)}(t)=O(t^{j+1+k+1})$, as $t\rightarrow0+$,
and using
\begin{align*}
\hspace{-7mm}
b_{-2(j+1)}^{(2k+2)}(t)=-\partial_{X}\left(\int_{0}^{t}e^{-(t-t^{'})a_{2}}\left(E(a_{2}^{(1)},b_{-2(j+1)}^{(2k)})+E(a_{2}^{(2)},b_{-2(j+1)}^{(2k-1)})+r_{-2j}^{(2k+1)}\right)\,dt^{'}\right),
\end{align*}
 also that
\[
b_{-2(j+1)}^{(2k+2)}(t)=O(t^{j+1+k+1}),\;t\rightarrow0+,
\]
 which proves the result for the case $h=0$.

Now, to complete the proof of this proposition we need to prove the
result for the case $h=1$, that is, for all $\ell=0$, $1$
\[
b_{-2(j+1)-1}=O(t^{j+2}),\,\,b_{-2(j+1)-1}^{(2k+1+\ell)}=O(t^{j+1+k+\ell+1}),\,t\rightarrow0+.
\]
To do it, we have to examine $r_{2-2(j+1)-1}$ and its derivatives,
that is,
\begin{align*}
r_{2-2(j+1)-1}= & a_{0}b_{-2j-1}+a_{1}b_{-2(j+1)}+\frac{1}{2}\left(\frac{-i}{2}\right)^{2}\{a_{2},b_{-2(j-1)-1}\}_{(2)}-\frac{i}{2}\{a_{2},b_{-2j-1}\}\\
 & -\frac{i}{2}\{a_{1},b_{-2j}\}\\
= & O(t^{j+1})+O(t^{j+1})+O(t^{j-1+1+1})+O(t^{j+1})+O(t^{j+1})\\
= & O(t^{j+1}),\;t\rightarrow0+,
\end{align*}
and 
\begin{align*}
r_{2-2(j+1)-1}^{(2k+1)}= & a_{0}b_{-2j-1}^{(2k+1)}+a_{1}b_{-2(j+1)}^{(2k+1)}+E(a_{1}^{(1)},b_{-2(j+1)}^{(2k)})+E(a_{2}^{(2)},b_{-2(j-1)-1}^{(2k+3)})\\
 & +E(a_{2}^{(1)},b_{-2j-1}^{(2k+2)})+E(a_{2}^{(2)},b_{-2j-1}^{(2k+1)})+E(a_{1}^{(1)},b_{-2j}^{(2k+2)})\\
= & O(t^{j+k+1})+O(t^{j+1+k+1})+O(t^{j+1+k-1+1})+O(t^{j-1+k+1+1})\\
 & +O(t^{j+k+1+1})+O(t^{j+k+1})+O(t^{j+k+1})\\
= & O(t^{j+k+1}),\;t\rightarrow0+.
\end{align*}
Taking an extra derivative, one immediately sees also that
\[
r_{2-2(j+1)-1}^{(2k+2)}=O(t^{j+k+2}),\;t\rightarrow0+.
\]
Hence, for all $\ell=0,1$ and for all $k\geq-1$
\[
r_{2-2(j+1)-1}^{(2k+1+\ell)}=O(t^{j+k+\ell+1}),\;t\rightarrow0+.
\]
(again, when $k=-1$ we take $\ell=1$). Since $b_{-2(j+1)-1}$ is
the solution of the Cauchy problem 
\begin{equation}
\begin{cases}
\partial_{t}b_{-2(j+1)-1}+a_{2}b_{-2(j+1)-1}=-r_{2-2(j+1)-1},\\
b_{-2(j+1)-1}|_{t=0}=0,
\end{cases}\label{eq:eq:Syst-b_-2(j+1)-1}
\end{equation}
we obtain $b_{-2(j+1)-1}(t)=O(t^{j+1+1})$ as $t\rightarrow0+$. As
before, taking one partial derivative with respect to $X$ yields
\[
\begin{cases}
\partial_{t}b_{-2(j+1)-1}^{(1)}+a_{2}b_{-2(j+1)-1}^{(1)}=-a_{2}^{(1)}b_{-2(j+1)-1}-r_{2-2(j+1)-1}^{(1)}=O(t^{j+2})+O(t^{j+1}),\\
b_{-2(j+1)-1}^{(1)}|_{t=0}=0,
\end{cases}
\]
whence it follows $b_{-2(j+1)-1}^{(1)}(t)=O(t^{j+2})$, and, taking
an extra derivative, also that, as $t\rightarrow0+$,
\[
b_{-2(j+1)-1}^{(2)}(t)=-\partial_{X}\left(\int_{0}^{t}e^{-(t-t^{'})a_{2}}\left(a_{2}^{(1)}b_{-2(j+1)-1}+r_{2-2(j+1)}^{(1)}\right)dt^{'}\right)=O(t^{j+3}).
\]
 Supposing then by induction the estimates up to order $2k$ and making
use of
\[
\hspace{-5mm}
\begin{cases}
\begin{array}{rl}
\partial_{t}b_{-2(j+1)-1}^{(2k+1)}+a_{2}b_{-2(j+1)-1}^{(2k+1)}= & -E(a_{2}^{(1)},b_{-2(j+1)-1}^{(2k)})-E(a_{2}^{(2)},b_{-2(j+1)-1}^{(2k-1)})-r_{2-2(j+1)-1}^{(2k+1)},\\
= & O(t^{j+1+k-1+1+1})+O(t^{j+1+k-1+1+1})+O(t^{j+1+k+1})
\end{array}\\
b_{-1}^{(2k+1)}|_{t=0}=0,
\end{cases}
\]
 we obtain $b_{-2(j+1)-1}^{(2k+1)}(t)=O(t^{j+k+2})$, as $t\rightarrow0+$,
and using
\begin{align*}
b_{-2(j+1)-1}^{(2k+2)}(t)= & -\partial_{X}\left(\int_{0}^{t}
e^{-(t-t^{'})a_{2}}\left(E(a_{2}^{(1)},b_{-2(j+1)-1}^{(2k)})+
E(a_{2}^{(2)},b_{-2(j+1)-1}^{(2k-1)})\right)\,dt^{'}\right)\\
 & -\partial_{X}\left(\int_{0}^{t}e^{-(t-t^{'})a_{2}}
\left(r_{2-2(j+1)-1}^{(2k+1)}\right)\,dt^{'}\right),
\end{align*}
 also that
\[
b_{-2(j+1)-1}^{(2k+2)}(t)=O(t^{j+1+k+1+1}),\;t\rightarrow0+,
\]
 which proves the proposition.

\end{proof}

\section{\label{sec:Meromorphic-continuation-of-zeta}Meromorphic continuation
of $\zeta_{A^{{\rm w}}}$}

Let $A^{\mathrm{w}}$ be as in the Section \ref{sec:Heat-semigroup-parametrix}.
In this section we will use the parametrix approximation of the heat-semigroup
construct in Lemma \ref{lem:Heat-Parametrix} to prove the result
about the continuation of the spectral zeta function of the class
of positive and self-adjoint elliptic operators $A^{\mathrm{w}}$
satisfying the hypotheses of Proposition \ref{prop:TimeDecrease-Diff}.
Namely, $\zeta_{A^{{\rm w}}}$ can be rewritten modulo a term holomorphic
on a an half plane of $\mathbb{C}$ as a linear complex combination
of meromorphic functions. Moreover, we will give explicit formulas
for the coefficients of this linear combination.

\begin{theorem} \label{Zeta_Diff_Thm}

Let $A=a_{2}+a_{1}+a_{0}$ be an elliptic system of second order where
$a_{j}$ is an $N\times N$  matrix-valued function on $\mathbb{R}^{2n}$
with homogeneous polynomial of degree $j$ entries for all $j=0,1,2$.
Moreover, suppose $A^{\mathrm{w}}>0$.

There exist constants $c_{-2j-h,n}$ with $0\leq j\leq n-1$, $h=0,1$,
and constants $c_{-2j-1,n}$, $C_{-2j}$ with $j\geq n$, such that,
for any given integer $\nu\in\mathbb{Z}_{+}$ with $\nu\geq n$,
\begin{align}
\zeta_{A^{{\rm w}}}(s)= & \frac{1}{\varGamma(s)}\left[\left(\sum_{h=0}^{1}\sum_{j=0}^{n-1}\frac{c_{-2j-h,n}}{s-(n-j)+h/2}\right)+\left(\sum_{j=n}^{\nu}\frac{c_{-2j-1,n}}{s-(n-j)+1/2}\right)\right.\label{eq:ZetaExpansion}\\
 & \left.+\left(\sum_{j=n}^{\nu}\frac{C_{-2j}}{s-(n-j)}\right)+H_{\nu}(s)\right]\nonumber 
\end{align}
where $\varGamma(s)$ is the Euler gamma function, and $H_{\nu}$
is holomorphic in the region ${\rm Re}s>(n-\nu)-1$. Consequently,
the spectral zeta function $\zeta_{A^{{\rm w}}}$ is meromorphic in
the whole complex plane $\mathbb{C}$ with at most simple poles at
$s=n$, $n-\frac{1}{2}$, $n-1$,$\ldots$ , $\frac{1}{2}$,$-\frac{1}{2}$,
$-\frac{3}{2}$, $...$, $n-\nu-\frac{1}{2}$. One has
\begin{equation}
c_{-2j-h,n}=(2\pi)^{-n}\int_{0}^{+\infty}\int_{\mathbb{S}^{2n-1}}\mathsf{Tr}\left(b_{-2j-h}(\rho^{2},\omega)\right)\rho^{2(n-j)-1-h}\,d\omega\,d\rho,\label{eq:c-Definition}
\end{equation}
where $0\leq j\leq n-1$, $h=0,1$ or $j\geq n$, $h=1$. In (\ref{eq:c-Definition})
the $b_{-2j-h}$ are the terms in the symbol of the parametrix $U_{A}\in{\rm OP}S_{{\rm sreg}}(2,0)$
constructed in the proof of Lemma \ref{lem:Heat-Parametrix} and Remark
\ref{rem:ConstructB_-j},
\[
U_{A}\sim{\displaystyle \sum_{j\geq0}}B_{-j}.
\]

\end{theorem}

\begin{proof}

The proof follows they idea to make use of the asymptotic expansion
given by Lemma \ref{lem:Heat-Parametrix} to obtain an asymptotic
expansion for the continuation of $\zeta_{A^{{\rm w}}}.$To do that
we write $\zeta_{A^{{\rm w}}}$ by the Mellin transform which gives
it in terms of the heat-semigroup of $A^{\mathrm{w}}$. Now the semi-group
can be approximated via Lemma \ref{lem:Heat-Parametrix}. Hence, we
compute an approximation of $\zeta_{A^{{\rm w}}}$ whose asymptotic
terms, given by integrals, are the $\frac{c_{-2j-h,n}}{s-(n-j)+h/2}$
in (\ref{eq:ZetaExpansion}), obtained by a Taylor expansion argument.
Actually, we need the integrals defining the $c_{-2j-h,n}$ to converge.
That is why we use Proposition \ref{prop:TimeDecrease-Diff} to have
a control on the vanishing of the asymptotic terms of the parametrix
of the heat-semigroup as $t\rightarrow0+$. Finally, we take into
account the residuals given by the approximations made and we sum
their contributes. Namely, we notice that they do not affect the values
of the $c_{-2j-h,n}$ for $j\leq n-1$ and those for $h=1$ if $j\geq n.$ 

By the properties of the heat semi-group $0\leq t\rightarrow e^{-tA^{\mathrm{w}}}$
we may use the Mellin transform and write

\[
(A^{{\rm w}})^{-s}=\frac{1}{\Gamma(s)}\int_{0}^{+\infty}t^{s-1}e^{-tA^{{\rm w}}}\:dt,\,\,{\rm Re}s>2n/2=n,
\]

\noindent so that 
\[
s\mapsto\zeta_{A^{{\rm w}}}(s)={\rm Tr}(A^{{\rm w}})^{-s}=\frac{1}{\Gamma(s)}\int_{0}^{+\infty}t^{s-1}\mathrm{Tr}\:e^{-tA^{{\rm w}}}\:dt.
\]

\noindent Let hence $U_{A}\sim{\displaystyle \sum_{j\geq0}}B_{-j}\in{\rm {\rm OP}}S_{{\rm sreg}}(2,0)$
be the parametrix approximation of $e^{-tA^{{\rm w}}}$ constructed
in Lemma \ref{lem:Heat-Parametrix}. We write
\[
\zeta_{A^{{\rm w}}}(s)=\frac{1}{\varGamma(s)}\left(\int_{0}^{1}+\int_{1}^{+\infty}\right)t^{s-1}\mathrm{Tr}\:e^{-tA^{{\rm w}}}\,dt=:Z_{0}(s)+Z_{\infty}(s).
\]

\noindent In the first place we claim that $Z_{\infty}(s)$ is holomorphic
in $\mathbb{C}$. In fact, on the one hand, since $t\mapsto{\rm Tr}R(t)$
is rapidly decreasing for $t\rightarrow+\infty$ (where $R(t):=e^{-tA^{{\rm w}}}-U_{A}(t)$),
we have that for all $p\in\mathbb{N}$ and for all $t\geq1$
\[
|{\rm Tr}\:R(t)|\lesssim t^{-p}.
\]

\noindent On the other, given any $\nu\geq0$ and any symbol $b\in S(2,-2\nu)$,
we have (by definition of the class $S(\mu,\nu)$ at p. 79 of \cite{P1:2010})
that for all $t\geq1$ and all $p\in\mathbb{N}$
\begin{align*}
\left|(2\pi)^{-n}\int_{\mathbb{R}^{2n}}\mathsf{Tr}\:b(t,X)\,dX\right|= & \left|(2\pi)^{-n}\int_{0}^{+\infty}\int_{\mathbb{S}^{2n-1}}\mathsf{Tr}\:b(t,\rho\omega)\rho^{2n-1}\,d\omega\,d\rho\right|\\
= & t^{-p}\left|(2\pi)^{-n}\int_{0}^{+\infty}\int_{\mathbb{S}^{2n-1}}t^{p}\mathsf{Tr}\:b(t,\rho\omega)\rho^{2n-1}\,d\omega\,d\rho\right|\\
\lesssim & t^{-p}\int_{0}^{+\infty}\frac{\rho^{2n+1}}{(1+\rho)^{2\nu+2p}}\,d\rho\\
\lesssim & t^{-p}.
\end{align*}
\noindent (Here, we uses the polar coordinates $0\neq X=|X|{\displaystyle \frac{X}{|X|}}$
with $\rho\in\mathbb{R}_{+}$, $\omega\in\mathbb{S}^{2n-1}$, and
$d\omega$ is the induced Riemann measure on $\mathbb{S}^{2n-1}$.)
It thus follows that for all $p\in\mathbb{N}$ and for all $t\geq1$
\[
|{\rm Tr}\:U_{A}(t)|\lesssim t^{-p}.
\]
In conclusion, since
\[
{\rm Tr}\:e^{-tA^{{\rm w}}}={\rm Tr}\:U_{A}(t)+{\rm Tr}\:R(t),
\]
for every $p\geq1$ there exists $C_{p}>0$ such that
\[
|\mathrm{Tr}\:e^{-tA^{{\rm w}}}|\leq C_{p}t^{-p},\,\,\forall t\geq1,
\]
which proves the claim, since the term $1/\varGamma(s)$ is already
holomorphic in $\mathbb{C}$. Therefore the crux of the matter lies
in the study of the function $Z_{0}(s)$. To study it we need a better
understanding of the terms $\mathrm{Tr}\:B_{-2j-h}$, $j\geq0$, $h=0,1$.
Hence, we recall that by the homogeneity of the $b_{-2j-h}$, for
$t>0$, $j\geq0$, and $h=0,1$

\begin{align*}
\mathrm{Tr}\:B_{-2j-h}(t)= & (2\pi)^{-n}\int_{\mathbb{R}^{2n}}\chi(X)\mathsf{Tr}\left(b_{-2j-h}(t,X)\right)\,dX\\
= & (2\pi)^{-n}\int_{0}^{+\infty}\int_{\mathbb{S}^{2n-1}}\chi(\rho\omega)\mathsf{Tr}\left(b_{-2j-h}(t,\rho\omega)\right)\rho^{2n-1}\,d\omega\,d\rho\\
= & (2\pi)^{-n}\int_{0}^{+\infty}\int_{\mathbb{S}^{2n-1}}\chi(\rho\omega)\mathsf{Tr}\left(b_{-2j-h}(\rho^{2}t,\omega)\right)\rho^{2(n-j)-1-h}\,d\omega\,d\rho.
\end{align*}

\noindent We consider
\[
c_{-2j-h,n}:=(2\pi)^{-n}\int_{0}^{+\infty}\int_{\mathbb{S}^{2n-1}}\mathsf{Tr}\left(b_{-2j-h}(\rho^{2},\omega)\right)\rho^{2(n-j)-1-h}\,d\omega\,d\rho.
\]

\noindent We claim that
\[
|c_{-2j-h,n}|<+\infty,\,\,\forall j\in\mathbb{Z}_{+},\,h=0,1.
\]

\noindent In fact, the integral is convergent at $\rho=+\infty$
for all $j$ since $\mathsf{Tr}\left(b_{-2j-h}(\cdot,\omega)\right)$
is a Schwartz function, it is clearly convergent at $\rho=0$ for
$0\leq2j+h\leq2n-1$, and finally it is convergent at $\rho=0$ also
when $2j+h\geq2n$, for the singularity at $0$ of the factor $\rho^{2(n-j)-1-h}$
is compensated by $\mathsf{Tr}\left(b_{-2j-h}(t,\omega)\right)=O(t^{j+h})$
as $t\rightarrow0+$. 

We define now the function

\[
f_{-2j-h}(t):=-(2\pi)^{-n}\int_{0}^{1}\int_{\mathbb{S}^{2n-1}}(1-\chi(\rho\omega))\mathsf{Tr}\left(b_{-2j-h}(t,\rho\omega)\right)\rho^{2n-1}\,d\omega\,d\rho,\,\,j\in\mathbb{Z}_{+}.
\]
 Then $f_{-j^{'}}\in C^{\infty}([0,+\infty);\mathbb{C})$, for all
$j^{'}\in\mathbb{Z}_{+}$, and by Proposition \ref{prop:TimeDecrease-Diff}
\begin{equation}
f_{-2j-h}(t)=O(t^{j+h}),\,\,t\rightarrow0+.\label{eq:f_Decrease}
\end{equation}
 It follows that
\begin{equation}
\mathrm{Tr}\:B_{-2j-h}(t)=c_{-2j-h,n}t^{-(n-j)+h/2}+f_{-2j-h}(t)=c_{-2j-h,n}t^{-(n-j)+h/2}+O(t^{j+h}),\label{eq:TrB}
\end{equation}
as $t\rightarrow0+$, for all $j\geq0$, $h=0,1$, and that (by the proof of Proposition
3.2.15 at p. 32 of \cite{P1:2010} adapted to the present setting),
\[
\mathrm{Tr}\:U_{A}(t)-\sum_{h=0}^{1}\sum_{j=0}^{\nu}\mathrm{Tr}\:B_{-2j-h}(t)=:\mathrm{Tr}\:R_{2\nu+2}(t)=O(t^{\nu+1}),\,\,t\rightarrow0+,
\]
 $\forall\nu\in\mathbb{Z}_{+},\,h=0,1$.

\noindent However, the information contained in (\ref{eq:TrB}) alone
is not yet sufficient to obtain the continuation of $\zeta_{A^{{\rm w}}}$,
and we need a better control of $f_{-2j-h}$ . Notice that for all
$j,k\in\mathbb{Z}_{+}$, denoting $\partial_{t}^{k}f_{-2j-h}(t)$
by $f_{-2j-h}^{(k)}(t)$,
\[
f_{-2j-h}^{(k)}(t)=-(2\pi)^{-n}\int_{0}^{1}\int_{\mathbb{S}^{2n-1}}(1-\chi(\rho\omega))\mathsf{Tr}\left(\partial_{t}^{k}b_{-2j-h}(t,\rho\omega)\right)\rho^{2n-1}\,d\omega\,d\rho,
\]
so that $f_{-2j-h}^{(k)}(0)$ is finite and can be computed through
(\ref{eq:r-Definition-Diff}), and through the differential equations
(\ref{eq:Syst-b_0-Diff}), (\ref{eq:Syst-b_-1}), (\ref{eq:Syst-b_-2(j+1)}),
and (\ref{eq:eq:Syst-b_-2(j+1)-1}) used to construct the $b_{-2j-h}$.
Note, in particular, that 
\[
f_{0}^{(k)}(0)=(-1)^{k+1}(2\pi)^{-n}\int_{0}^{1}\int_{\mathbb{S}^{2n-1}}(1-\chi(\rho\omega))\mathsf{Tr}\left(a_{2}(\rho\omega)^{k}\right)\rho^{2n-1}\,d\omega\,d\rho.
\]
 We next apply Lemma 7.2.3 at p. 99 of \cite{P1:2010} to the functions
$f_{-2j-h}$, so that for any given $\nu\in\mathbb{Z}_{+}$ we may
write, by (\ref{eq:f_Decrease}), 
\[
F_{-2j-h}(s):=\int_{0}^{1}t^{s-1}f_{-2j-h}(t)\,dt=\sum_{k=0}^{\nu}\frac{f_{-2j-h}^{(j+h+k)}(0)}{(j+h+k)!}\frac{1}{s+j+h+k}+F_{-2j-h,\nu}(s),
\]
 where $F_{-2j-h,\nu}$ is holomorphic for ${\rm Re}s>-j-h-\nu-1$.

Using this in (\ref{eq:TrB}) we have that for each $j\geq0$, $h=0,1$,
for any given $\nu\in\mathbb{Z}_{+}$,
\begin{align*}
s\mapsto\int_{0}^{1}t^{s-1}\mathrm{Tr}\:B_{-2j-h}(t)\,dt= & \frac{c_{-2j-h,n}}{s-(n-j)+h/2}\\
 & +\left(\sum_{k=0}^{\nu}\frac{f_{-2j-h}^{(j+h+k)}(0)}{(j+h+k)!}\frac{1}{s+j+h+k}\right)+F_{-2j-h,\nu}(s),
\end{align*}
 where $F_{-2j-h,\nu}$ is holomorphic for ${\rm Re}s>-j-h-\nu-1$.

Analogously, since $0\leq t\mapsto\mathrm{Tr}\:R(t)\in\mathscr{Sb}(\mathbb{\overline{R}}_{+};\mathbb{C})$
and by Lemma 7.2.3 at p. 99 of \cite{P1:2010} we also have, with
$f_{R}(t):=\mathrm{Tr}\:R(t)$, that for any given $\nu\in\mathbb{Z}_{+}$
\[
\int_{0}^{1}t^{s-1}f_{R}(t)\,dt=\sum_{k=0}^{\nu}\frac{f_{R}^{(k)}(0)}{k!}\frac{1}{s+k}+F_{R,\nu}(s),
\]
 where $F_{R,\nu}$ is holomorphic for ${\rm Re}s>-\nu-1$.

We therefore obtain that for any given $\nu\in\mathbb{Z}_{+}$.
\begin{align*}
Z_{0}(s)= & \frac{1}{\Gamma(s)}\left[\left(\sum_{h=0}^{1}\sum_{j=0}^{\nu}\int_{0}^{1}t^{s-1}\mathrm{Tr}\:B_{-2j-h}(t)\:dt\right)\right.\\
 & +\int_{0}^{1}t^{s-1}\mathrm{Tr}\:R_{2\nu+2}(t)\:dt+\int_{0}^{1}t^{s-1}\mathrm{Tr}\:R(t)\:dt\Biggr]
\end{align*}
 Since the function $s\mapsto\int_{0}^{1}t^{s-1}{\rm Tr}R_{2\nu+2}(t)\,dt=:F_{2\nu+2}(s)$
is holomorphic for ${\rm Re}s>-\nu-1$, we thus obtain that, for any
given $\nu\in\mathbb{Z}_{+}$ with $\nu\geq n$,
\begin{align*}
Z_{0}(s)=\frac{1}{\varGamma(s)} & \left[\sum_{h=0}^{1}\sum_{j=0}^{\nu}\frac{c_{-2j-h,n}}{s-(n-j)+h/2}+\left(\sum_{h=0}^{1}\sum_{j,k=0}^{\nu}\frac{f_{-2j-h}^{(j+h+k)}(0)}{(j+h+k)!}\frac{1}{s+j+h+k}\right)\right.\\
 & +\left.\sum_{k=0}^{\nu}\frac{f_{R}^{(k)}(0)}{k!}\frac{1}{s+k}+\left(\sum_{h=0}^{1}\sum_{j=0}^{\nu}F_{-2j-h,\nu}(s)\right)+F_{R,\nu}(s)+F_{2\nu+2}(s)\right]\\
=\frac{1}{\varGamma(s)} & \Biggl[\left(\sum_{h=0}^{1}\sum_{j=0}^{n-1}\frac{c_{-2j-h,n}}{s-(n-j)+h/2}\right)+\left(\sum_{j=n}^{\nu}\frac{c_{-2j-1,n}}{s-(n-j)+1/2}\right)\\
 & +\left(\sum_{j=n}^{\nu}\frac{C_{-2j}}{s-(n-j)}\right)+\tilde{H}_{\nu}(s)\Biggr],
\end{align*}
 with $s\mapsto\tilde{H}_{\nu}(s)$ holomorphic for ${\rm Re}s>(n-\nu)-1$.
Since the function $1/\varGamma(s)$ is holomorphic in $\mathbb{C}$
and has zeros at the non-positive integers $-k,\,k\in\mathbb{Z}_{+}$,
this proves the theorem.

\end{proof}

\begin{remark}

An interesting problem can be to use the asymptotics for resolvent
expansions and trace regularizations by \cite{Hi-Po1:2003} and \cite{Hi-Po2:2002}.

\end{remark}

Theorem \ref{Zeta_Diff_Thm} has the following corollary for the Hurwitz-type
spectral zeta function of $A^{\mathrm{w}}$.

\begin{corollary}

Let $A=a_{2}+a_{1}+a_{0}$ be an elliptic system of second order where
$a_{j}$ is an $N\times N$  matrix-valued function on $\mathbb{R}^{2n}$
with homogeneous polynomial of degree $j$ entries for all $j=0,1,2$.
Moreover, suppose $A^{\mathrm{w}}>0$.

For all $\tau>0$ there exist constants $c_{-2j-h,n}$ with $0\leq j\leq n-1$,
$h=0,1$, and constants $c_{-2j-1,n}$, $C_{-2j}$ with $j\geq n$,
such that, for any given integer $\nu\in\mathbb{Z}_{+}$ with $\nu\geq n$,
\begin{align}
\zeta_{A^{{\rm w}}+\tau I}(s)= & \frac{1}{\varGamma(s)}\left[\left(\sum_{h=0}^{1}\sum_{j=0}^{n-1}\frac{c_{-2j-h,n}}{s-(n-j)+h/2}\right)+\left(\sum_{j=n}^{\nu}\frac{c_{-2j-1,n}}{s-(n-j)+1/2}\right)\right.\label{eq:ZetaExpansion-Hurwitz}\\
 & \left.+\left(\sum_{j=n}^{\nu}\frac{C_{-2j}}{s-(n-j)}\right)+H_{\nu}(s)\right]\nonumber 
\end{align}
where $\varGamma(s)$ is the Euler gamma function, and $H_{\nu}$
is holomorphic in the region ${\rm Re}s>(n-\nu)-1$. Consequently,
the spectral zeta function $\zeta_{A^{{\rm w}}}$ is meromorphic in
the whole complex plane $\mathbb{C}$ with at most simple poles at
$s=n$, $n-\frac{1}{2}$, $n-1$,$\ldots$ , $\frac{1}{2}$,$-\frac{1}{2}$,
$-\frac{3}{2}$, $...$, $n-\nu-\frac{1}{2}$. One has
\begin{align}
\hspace{-5mm}
& c_{-2j-h,n}= \nonumber \\
& (2\pi)^{-n}\int_{0}^{+\infty}\int_{\mathbb{S}^{2n-1}}\mathsf{Tr}\left(b_{-2j-h}(\rho^{2},
\omega)\right)\rho^{2(n-j)-1-h}\,d\omega d\rho\label{eq:c-Definition-Hurwitz}\\
 & -\tau(2\pi)^{-n}\int_{0}^{+\infty}\int_{\mathbb{S}^{2n-1}}\int_{0}^{\rho^{2}}e^{-(\rho^{2}-t^{'})a_{2}}
\mathsf{Tr}(b_{2-2j-h}(t^{'},\omega))\rho^{2(n-j)-1-h}\,dt^{'} d\omega d\rho \nonumber
\end{align}
where $0\leq j\leq n-1$, $h=0,1$ or $j\geq n$, $h=1$. In (\ref{eq:c-Definition-Hurwitz})
the $b_{-2j-h}$ are the terms in the symbol of the parametrix $U_{A}\in{\rm OP}S_{{\rm sreg}}(2,0)$
constructed in the proof of Lemma \ref{lem:Heat-Parametrix} and Remark
\ref{rem:ConstructB_-j},
\[
U_{A}\sim{\displaystyle \sum_{j\geq0}}B_{-j},
\]
where we set $b_{k}\equiv0$ for all $k=1$, $2$.

\end{corollary}

\begin{proof}

The proof follows from the demonstration of Theorem \ref{Zeta_Diff_Thm}.
In fact, we use of the equations in the proof of Lemma \ref{lem:Heat-Parametrix}
and (\ref{eq:b-Formula-Diff}) to link the asymptotic expansion of
the parametrix of the heat semi-group of $A^{{\rm w}}+\tau I$ to
the one of $A^{{\rm w}}$. Let $b_{j}$, $r_{2-j}$ be the terms constructed
in the proof of Lemma \ref{lem:Heat-Parametrix} (see also (\ref{eq:r-Definition-Diff})
and (\ref{eq:b-Formula-Diff})) for $A^{{\rm w}}$ and $\tilde{b}_{-j}$,
$r_{2-j}$ those for $A^{{\rm w}}+\tau I$. Then,
\begin{equation}
\begin{cases}
\tilde{b}_{0}(t,X)=e^{-ta_{2}(X)},\\
\tilde{b}_{1}(t,X)=\int_{0}^{t}e^{-(t-t^{'})a_{2}}r_{2-j}(t^{'},X)\,dt^{'}\\
\tilde{b}_{-j}(t,X)=-\int_{0}^{t}e^{-(t-t^{'})a_{2}}r_{2-j}(t^{'},X)\,dt^{'}-\tau\int_{0}^{t}e^{-(t-t^{'})a_{2}}b_{2-j}(t^{'},X)\,dt^{'},\:j\geq2,
\end{cases}\label{eq:b-Definition-Hurwitz}
\end{equation}
 since for all $j\geq2$
\[
\begin{cases}
\frac{d}{dt}\tilde{b}_{-j}+a_{2}\tilde{b}_{-j}=-\tilde{r}_{2-j}=-r_{2-j}-\tau b_{2-j},\\
\tilde{b}_{-j}|_{t=0}=0.
\end{cases}
\]

Now, we apply Theorem \ref{Zeta_Diff_Thm} to $\zeta_{A^{{\rm w}}+\tau I}$,
obtaining (\ref{eq:ZetaExpansion-Hurwitz}) with coefficients
\begin{equation}
c_{-2j-h,n}=(2\pi)^{-n}\int_{0}^{+\infty}\int_{\mathbb{S}^{2n-1}}\mathsf{Tr}\left(\tilde{b}_{-2j-h}(\rho^{2},\omega)\right)\rho^{2(n-j)-1-h}\,d\omega\,d\rho.\label{eq:c-Def-Tilde-Hurwitz}
\end{equation}
 Actually, substituting in (\ref{eq:c-Def-Tilde-Hurwitz}) the expressions
for $\tilde{b}_{-j}$ given by (\ref{eq:b-Definition-Hurwitz}), we
obtain (\ref{eq:c-Definition-Hurwitz}) which completes the proof.

\end{proof}

\section{\label{sec:Example} Examples}

\subsection{The meromorphic continuation of Jayne-Cumming model spectral zeta
function ($n=1$, $N=2$).}

The Jaynes-Cumming (JC) model is the model of a two-level atom in
one cavity, given by the $2\times2$ system in one real variable $x\in\mathbb{R}$
(see 3.1 in \cite{MaPa:2022})
\[
A^{\mathrm{w}}(x,D)=\alpha p_{2}^{\mathrm{w}}(x,D)I_{2}+\beta\Bigl(\boldsymbol{\sigma}_{+}\psi^{\mathrm{w}}(x,D)^{*}+\boldsymbol{\sigma}_{-}\psi^{\mathrm{w}}(x,D)\Bigr)+\gamma\boldsymbol{\sigma}_{3},\;\alpha>0,\beta,\gamma\in\mathbb{R},
\]
where $\psi(x,D):=\frac{x+\partial_{x}}{\sqrt{2}}$, $\boldsymbol{\sigma}_{\pm}:=\frac{1}{2}(\boldsymbol{\sigma}_{1}\pm i\boldsymbol{\sigma}_{2})$
with $\boldsymbol{\sigma}_{j},$ $j=0,\ldots,3$, the Pauli-matrices,
i.e. 
\[
\boldsymbol{\sigma}_{0}:=I_{2},\quad\boldsymbol{\sigma}_{1}=\left[\begin{array}{cc}
0 & 1\\
1 & 0
\end{array}\right],\quad\boldsymbol{\sigma}_{2}:=\left[\begin{array}{cc}
0 & -i\\
i & 0
\end{array}\right],\quad\boldsymbol{\sigma}_{3}:=\left[\begin{array}{cc}
1 & 0\\
0 & -1
\end{array}\right],
\]
and the atom levels are given by $\pm\gamma$. 

To apply Theorem \ref{Zeta_Diff_Thm} we need to compute the terms
$b_{-j}$ of the asymptotic expansion of the semi-group parametrix
construct in Lemma \ref{lem:Heat-Parametrix}. First of all, if $A$
is a the Hamiltonian of the JC model and, in the notations of the
previous sections, $A=a_{2}+a_{1}+a_{0}$, then 
\begin{equation}
a_{1}a_{0}=-a_{0}a_{1},\:a_{0}^{2}=I_{2},\text{ and }a_{1}^{2}=p_{2},.\label{eq:RulesJC}
\end{equation}
where $p_{2}$ is the hrmonic oscillator symbol. Hence, the product
of any number of factors equal to $a_{1}$ or $a_{0}$ can be rewritten
as the multiple (by a function in $C^{\infty}(\mathbb{R}_{t};C^{\infty}(\mathbb{R}^{2n}))$)
of $a_{1}$, $a_{0},$$a_{0}a_{1}$ or $I_{2}$ by using iteratively
the identities (\ref{eq:RulesJC}). This fact motivates the following
definition.

\begin{definition}

Given a linear combinations of products of any number of $a_{0}$
and $a_{1}$, we say that it is written in \textit{irreducible form}
if it is a linear combination of $a_{1}$, $a_{0}$, $a_{0}a_{1}$
and $I_{2}$ with coefficients in $C^{\infty}(\mathbb{R}_{t};C^{\infty}(\mathbb{R}^{2n}))$.

\end{definition}

We are going to prove a lemma determining the structure of the $b_{j}$
as linear combination with coefficients in $C^{\infty}(\mathbb{R}_{t};C^{\infty}(\mathbb{R}^{2n}))$
of $a_{1}$, $a_{0},$$a_{0}a_{1}$ and $I_{2}$.

\begin{lemma} \label{lem:JC_Parametrix}

Let $A=a_{2}+a_{1}+a_{0}$ be the Hamiltonian of the JC model with
$a_{j}$ homogeneous of degree $j$. Then, the $b_{-j}$ can be written
in irreducible form. Moreover, 
\begin{equation}
j\text{ odd }\Rightarrow\text{ the coefficients of }a_{0},\,I_{2}\text{ in the irreducible form of }b_{-j}\text{ are }0,\label{eq:JC_odd}
\end{equation}
\begin{equation}
j\text{ even }\Rightarrow\text{ the coefficients of }a_{1},\,a_{0}a_{1}\text{ in the irreducible form of }b_{-j}\text{ are }0.\label{eq:JC_even}
\end{equation}

\end{lemma}

\begin{proof}

The proof is by induction, follows the construction of the parametrix
in Lemma \ref{lem:Heat-Parametrix} and here we will use the same
notations of that lemma. First of all, 
\[
b_{0}(t,X)=e^{-tp_{2}(X)}I_{2},\quad b_{-1}(t,X)=-te^{-tp_{2}(X)}a_{1}
\]
(see also (\ref{eq:b-Formula-Diff})). Hence, $b_{0}$ and $b_{-1}$
are already written in irreducible form and satisfy (\ref{eq:JC_odd})
and (\ref{eq:JC_even}). Now, we suppose that for all $j^{'}\leq2j-1$
($j\geq2$) the thesis is verified and we want to prove the result
for $b_{2j}$ and $b_{2j+1}$. By the construction in Lemma \ref{lem:Heat-Parametrix}
and since $A$ is a differential operator (that is, its expansion
contains only terms with degree of homogeneity $\geq0$), 
\[
\begin{cases}
\frac{d}{dt}b_{-2j}+p_{2}b_{-2j}=-a_{0}b_{2-2j}-a_{1}b_{1-2j},\\
b_{-2j}|_{t=0}=0.
\end{cases}
\]
 Hence, since by inductive hypothesis
\begin{align*}
a_{0}b_{2-2j}+a_{1}b_{1-2j}= & a_{0}(f_{1}a_{0}+f_{2}I_{2})+a_{1}(g_{1}a_{1}+g_{2}a_{0}a_{1})\\
= & f_{1}a_{0}^{2}+f_{2}a_{0}+g_{1}a_{1}^{2}+g_{2}a_{1}a_{0}a_{1}\\
= & f_{1}I_{2}+f_{2}a_{0}-g_{1}p_{2}I_{2}-g_{2}p_{2}a_{0},
\end{align*}
where the third equality follows from (\ref{eq:RulesJC}) and where
the $f_{j}$ and $g_{j}$ are function in $C^{\infty}(\mathbb{R}_{t};C^{\infty}(\mathbb{R}^{2n}))$.
Hence, the claim is verified for $b_{-2j}$. Repeating the argument
for $b_{-2j-1}$, we have
\[
\begin{cases}
\frac{d}{dt}b_{-2j-1}+p_{2}b_{-2j-1}=-a_{0}b_{2-2j-1}-a_{1}b_{1-2j-1},\\
b_{-2j-1}|_{t=0}=0,
\end{cases}
\]
 which, since

\begin{align*}
a_{0}b_{1-2j}+a_{1}b_{-2j}= & a_{0}(\tilde{f}_{1}a_{1}+\tilde{f}_{2}a_{0}a_{1})+a_{1}(\tilde{g}_{1}a_{0}+\tilde{g}_{2}I_{2})\\
= & \tilde{f}_{1}a_{0}a_{1}+\tilde{f}_{2}a_{0}^{2}a_{1}+\tilde{g}_{1}a_{1}a_{0}+\tilde{g}_{2}a_{1}\\
= & \tilde{f}_{1}a_{0}a_{1}+\tilde{f}_{2}a_{1}-\tilde{g}_{1}a_{0}a_{1}+\tilde{g}_{2}a_{1},
\end{align*}
 shows that the claim is verified also for $b_{-2j-1}$ and completes
the proof.

\end{proof}

\begin{remark}

By Lemma \ref{lem:JC_Parametrix} we have that $\mathsf{Tr}\left(b_{-2j-1}\right)=0$
since it is a linear combination of matrices with zeros on the principal
diagonal. Hence, by (\ref{eq:JC_odd}) and (\ref{eq:c-Definition})
we have that 
\[
c_{-2j-1,1}=(2\pi)^{-1}\int_{0}^{+\infty}\int_{0}^{2\pi}\mathsf{Tr}\left(b_{-2j-1}(\rho^{2},\omega)\right)\rho^{-2j}\,d\omega\,d\rho=0,\:j\geq0,
\]
and
\begin{align*}
c_{0,1}= & (2\pi)^{-1}\int_{0}^{+\infty}\int_{0}^{2\pi}\mathsf{Tr}\left(b_{0}(\rho^{2},\omega)\right)\rho\,d\omega\,d\rho,\\
= & 2(2\pi)^{-1}\int_{0}^{+\infty}\int_{0}^{2\pi}e^{-\rho^{2}/2}\rho\,d\omega\,d\rho\\
= & 2\int_{0}^{+\infty}e^{-\rho^{2}/2}\rho\,d\rho=2.
\end{align*}
 Therefore, if $A$ is the JC Hamiltonian, by (\ref{eq:ZetaExpansion})
the spectral zeta function associated to $A^{\mathrm{w}}$ is 
\[
\zeta_{A^{\mathrm{w}}}(s)=\frac{1}{\varGamma(s)}\left[\frac{2}{s-1}+\left(\sum_{j=1}^{\nu}\frac{C_{-2j}}{s-(1-j)}\right)+H_{\nu}(s)\right],
\]
 where $\nu\geq1$, $H_{\nu}$ is holomorphic in the region ${\rm Re}s>-\nu$
and the $c_{-2,1}$, $C_{-2j}$ has been defined in Theorem \ref{Zeta_Diff_Thm}.
Consequently, the spectral zeta function $\zeta_{A^{{\rm w}}}$ is
meromorphic in the whole complex plane $\mathbb{C}$ with a simple
pole at $s=1$. Thus, $\zeta_{A^{\mathrm{w}}}$ has a meromorphic
continuation to $\mathbb{C}$.

\end{remark}

\subsection{The JC-model for one atom with $3$-level and one cavity-mode in
the so called \textit{$\Xi$-configuration}.}

This generalization of the JC model (that we will denote by $3$-$\Xi$-JCM)
describes a $3$-level atom in one cavity, given by the $3\times3$
system in one real variable $x\in\mathbb{R}^{2}$ (see 3.2 in \cite{MaPa:2022}).
In this configuration every level of energy can interact only with
the ones near to it, that is the electron can absorb (or emit) a photon
moving from the $j$th level of energy to the $j+1$st (or from the
$j+1$st level of energy to the $j$th) for $j=1$, $2$. That is
mathematically represented by the following Hamiltonian operator.
For $\alpha>0,$ $\beta_{1},\beta_{2}\in\mathbb{R}\setminus\{0\}$,
$\gamma_{1},\gamma_{2},\gamma_{3}\in\mathbb{R}$ with $\gamma_{1}<\gamma_{2}<\gamma_{3}$,
\begin{align*}
A^{\mathrm{w}}(x,D)= & \alpha p_{2}^{\mathrm{w}}(x,D)I_{3}+\frac{1}{2}\sum_{k=1}^{2}\beta_{k}\Bigl(\psi^{\mathrm{w}}(x,D)^{*}E_{k,k+1}+\psi^{\mathrm{w}}(x,D)E_{k+1,k}\Bigr)\\
 & +\sum_{k=1}^{3}\gamma_{k}E_{kk},
\end{align*}
with 
\[
E_{jk}:=e_{k}^{*}\otimes e_{j},\quad1\leq j,k\leq3
\]
forming the basis of the $3\times3$ complex matrices, where $E_{jk}$
acts on $\mathbb{C}^{3}$ as 
\[
E_{jk}w=\langle w,e_{k}\rangle e_{j},\quad w\in\mathbb{C}^{3},
\]
 and $\psi(x,D):=\frac{x+\partial_{x}}{\sqrt{2}}$.

\begin{lemma} \label{lem:JC-3_Parametrix}

Let $A=a_{2}+a_{1}+a_{0}$ be the Hamiltonian of the $3$-$\Xi$-JCM
with $a_{j}$ homogeneous of degree $j$. Then, 
\begin{equation}
j\text{ odd }\Rightarrow\text{the principal and secondary diagonal entries of }b_{-j}\text{ are }0,\label{eq:JC-3_odd}
\end{equation}
\begin{equation}
j\text{ even }\Rightarrow\text{the subdiagonal and superdiagonal entries of }b_{-j}\text{ are }0.\label{eq:JC-3_even}
\end{equation}

\end{lemma}

\begin{proof}

Again the proof is by induction, follows the construction of the parametrix
in Lemma \ref{lem:Heat-Parametrix} and here we will use the same
notations of that lemma. First of all, 
\[
b_{0}(t,X)=e^{-tp_{2}(X)}I_{2},\quad b_{-1}(t,X)=-te^{-tp_{2}(X)}a_{1}.
\]
Hence, $b_{0}$ and $b_{-1}$ satisfy (\ref{eq:JC-3_odd}) and (\ref{eq:JC-3_even}).
Now, we suppose that for all $j^{'}\leq2j-1$ ($j\geq2$) the thesis
is verified and we want to prove the result for $b_{2j}$ and $b_{2j+1}$.
By the construction in Lemma \ref{lem:Heat-Parametrix} and since
$A$ is a differential operator
\[
\begin{cases}
\frac{d}{dt}b_{-2j}+p_{2}b_{-2j}=-a_{0}b_{2-2j}-a_{1}b_{1-2j},\\
b_{-2j}|_{t=0}=0.
\end{cases}
\]
 Therefore, by inductive hypothesis $a_{0}b_{2-2j}$ subdiagonal and
superdiagonal entries are $0$ since $a_{0}$ is a diagonal matrix.
Moreover, $a_{1}b_{1-2j}$ subdiagonal and superdiagonal entries are
$0$ since the principal and secondary diagonal entries of $b_{1-2j}$
are $0$. Hence, the claim is verified for $b_{-2j}$. Repeating the
argument for $b_{-2j-1}$, we have
\[
\begin{cases}
\frac{d}{dt}b_{-2j-1}+p_{2}b_{-2j-1}=-a_{0}b_{2-2j-1}-a_{1}b_{1-2j-1},\\
b_{-2j-1}|_{t=0}=0.
\end{cases}
\]
Thus, by inductive hypothesis $a_{0}b_{1-2j}$ has principal and secondary
diagonal entries that are $0$ since $a_{0}$ is diagonal. Moreover,
$a_{1}b_{-2j}$ principal and secondary diagonal entries are $0$
since $b_{-2j}$ diagonal entries are $0$. Hence, the claim is verified
also for $b_{-2j-1}$.

\end{proof}

\begin{remark}

By Lemma \ref{lem:JC-3_Parametrix} we have that $\mathsf{Tr}\left(b_{-2j-1}\right)=0$
since $b_{-2j-1}$ principal diagonal entries are $0$. Hence, by
(\ref{eq:JC-3_odd}) and (\ref{eq:JC-3_even}) we have that 
\[
c_{-2j-1,1}=(2\pi)^{-1}\int_{0}^{+\infty}\int_{0}^{2\pi}\mathsf{Tr}\left(b_{-2j-1}(\rho^{2},\omega)\right)\rho^{-2j}\,d\omega\,d\rho=0,\:j\geq0,
\]
and
\begin{align*}
c_{0,1}= & (2\pi)^{-1}\int_{0}^{+\infty}\int_{0}^{2\pi}\mathsf{Tr}\left(b_{0}(\rho^{2},\omega)\right)\rho\,d\omega\,d\rho,\\
= & 3(2\pi)^{-1}\int_{0}^{+\infty}\int_{0}^{2\pi}e^{-\rho^{2}/2}\rho\,d\omega\,d\rho\\
= & 3\int_{0}^{+\infty}e^{-\rho^{2}/2}\rho\,d\rho=3.
\end{align*}
 Therefore, if $A$ is the $3$-$\Xi$-JCM Hamiltonian, by (\ref{eq:ZetaExpansion})
the spectral zeta function associated to $A^{\mathrm{w}}$ is 
\[
\zeta_{A^{\mathrm{w}}}(s)=\frac{1}{\varGamma(s)}\left[\frac{3}{s-1}+\left(\sum_{j=1}^{\nu}\frac{C_{-2j}}{s-(1-j)}\right)+H_{\nu}(s)\right],
\]
where $\nu\geq1$, $H_{\nu}$ is holomorphic in the region ${\rm Re}s>-\nu$
and the $c_{-2,1}$, $C_{-2j}$ has been defined in Theorem \ref{Zeta_Diff_Thm}.
Consequently, the spectral zeta function $\zeta_{A^{{\rm w}}}$ is
meromorphic in the whole complex plane $\mathbb{C}$ with a simple
pole at $s=1$. Thus, $\zeta_{A^{\mathrm{w}}}$ has a meromorphic
continuation to $\mathbb{C}$.

\end{remark}

\section*{Acknowledgements}

I wish to thank Prof. Alberto Parmeggiani for helpful discussions.

\end{document}